\documentclass[11pt]{article}
\usepackage[margin =1in]{geometry}
\usepackage{amssymb,amsthm,amsmath}
\usepackage{enumerate}
\usepackage{hyperref}
\usepackage{cite}
\usepackage[capitalize]{cleveref}
\usepackage{enumitem}
\usepackage{color}
\usepackage{cancel}
\usepackage{pgf}
\usepackage{svg}
\usepackage{pgfplots}
\usepackage{import}
\usepackage{upgreek}
\usepackage{fancyhdr}
\usepackage{setspace}

\usepackage[utf8]{inputenc} 
\usepackage[T1]{fontenc}    
\usepackage{hyperref}       
\usepackage{url}            
\usepackage{booktabs}       
\usepackage{amsfonts}       
\usepackage{nicefrac}       
\usepackage{microtype}
\usepackage{multirow}
\usepackage{xfrac}
\usepackage{todonotes}
\usepackage{titlesec}
\usepackage{caption}
\usepackage{subcaption}
\usepackage{algorithm}
\usepackage{algorithmic}

\titleformat{\subsubsection}[runin]
{\normalfont\normalsize\bfseries}{\thesubsubsection}{1em}{}

\usepackage{graphicx}

\usepackage{appendix}
\numberwithin{equation}{section}

\newcommand{\inclu}[0] {\ar@{^{(}->}}

\newtheorem{thm}{Theorem}[section]

\newtheorem{definition}[thm]{Definition}
\newtheorem{proposition}[thm]{Proposition}

\newtheorem{corollary}[thm]{Corollary}

\newtheorem{assumption}{Assumption}

\crefname{claim}{claim}{claims}
\Crefname{claim}{Claim}{Claims}
\crefname{lem}{lemma}{lemmas}
\Crefname{lem}{Lemma}{Lemmas}
\crefname{algorithm}{algorithm}{algorithms}
\Crefname{algorithm}{Algorithm}{Algorithms}

\theoremstyle{remark}

\usepackage{mathtools}

\usepackage[algo2e, boxruled]{algorithm2e}

\usepackage[T1]{fontenc}
\usepackage{lmodern}

\pgfplotsset{compat=1.14}

\makeatletter
\def\thanks#1{\protected@xdef\@thanks{\@thanks
        \protect\footnotetext{#1}}}
\makeatother

\begin{document}
    \title{Switch Updating in SPSA Algorithm for\\Stochastic Optimization with Inequality Constraints}

    \author{Zhichao Jia$^{1}$ \qquad Ziyi Wei$^{1}$ \qquad James C. Spall$^{2}$}
    \thanks{$^{1}$Zhichao Jia is with the Johns Hopkins University, Department of Applied Mathematics and Statistics (zjia12@jhu.edu).}%
    \thanks{$^{1}$Ziyi Wei is with the Johns Hopkins University, Department of Applied Mathematics and Statistics (zwei19@jhu.edu).}%
    \thanks{$^{2}$James C. Spall is with the Johns Hopkins University, Applied Physics Laboratory and Department of Applied Mathematics and Statistics (james.spall@jhuapl.edu).}%
	\date{}
	\maketitle
	\begin{abstract}
        Simultaneous perturbation stochastic approximation (SPSA) is widely used in stochastic optimization due to its high efficiency, asymptotic stability, and reduced number of required loss function measurements. However, the standard SPSA algorithm needs to be modified to deal with constrained problems. In recent years, sequential quadratic programming (SQP)-based projection ideas and penalty ideas have been analyzed. Both ideas have convergence results and a potentially wide range of applications, but with some limitations in practical consideration, such as computation time, complexity, and feasibility guarantee. We propose an SPSA-based switch updating algorithm, which updates based on the loss function or the inequality constraints, depending on current feasibility in each iteration. We show convergence results for the algorithm, and analyze its properties relative to other methods. We also numerically compare the switch updating algorithm with the penalty function approach for two constrained examples.
    \end{abstract}

    \section{Introduction}


Stochastic approximation (SA) is a family of stochastic algorithms finding the extremums of loss functions with only noisy measurements of them being available. To solve unconstrained stochastic optimization problems, several SA methods have been proposed. The simultaneous perturbation stochastic approximation (SPSA) algorithm analyzed in \cite{spall1992multivariate} is an efficient first-order stochastic approximation that requires only two noisy loss function measurements in each iteration. Applications of the SPSA algorithm can be found in \cite{9198090,hutchison2001simulation,xing2005application} and many other sources. However, the classic SPSA algorithm needs to be modified when is to be applied to constrained problems \cite{spall1993model}. This paper proposes a simple but effective and feasibility-guaranteed algorithm based on the SPSA algorithm to deal with some kinds of inequality constraints. The algorithm relies on the ideas of switching updating rules \cite{articlepolyak} according to different feasibility conditions.

Let us now describe our problem setting. We consider the following constrained optimization problem:
\begin{align}
    \min \ &L(\uptheta), \nonumber\\
    \mathrm{s.t.} \ &q_i(\uptheta) \leq 0, \qquad i=1,...,m,
    \label{problem}
\end{align}
where $L:\mathbb{R}^p \to \mathbb{R}$ is the objective loss function, and $q_i:\mathbb{R}^p \to \mathbb{R}, i=1,...,m$ are differentiable constraint functions. For $L(\uptheta)$, only noisy function measurements $y(\uptheta) = L(\uptheta)+\upvarepsilon(\uptheta)$ are available, where $\upvarepsilon(\uptheta)$ represents the noise term. To ensure a $p$-dimensional feasible region $\Theta=\{\uptheta|q_i(\uptheta) \leq 0, i=1,...,m\}$, we assume that the interior of $\Theta$ is not an empty set. The solution $\uptheta^*$ is assumed to be unique and may lie in anywhere in $\Theta$ (interior or boundary).

Many technical approaches in \cite{wright1999numerical} have proved useful in constrained problem settings. Projection to the feasible set $\Theta$ after updating in each iteration is a common way to maintain feasibility, and, in some cases, it can be simply realized in the SPSA algorithm. For example, \cite{sadegh1997constrained} and \cite{fu_optimization_1997} applied the idea of projecting current infeasible points back to their nearest points in $\Theta$. However, constraint regions are often sufficiently complex so that no way exists to locate exact projection points during iterations, making the projection idea impractical. More recently, \cite{shi2021sqp} applied a sequential quadratic programming (SQP)-based projection method to the SPSA algorithm, searching for approximate projection points by solving an equivalently converted deterministic quadratic constrained problem with the SQP algorithm in each iteration. With approximate projection points converging to exact projections during iterations, the modified projection SPSA algorithm shows convergence as well. This makes the projection idea useful when facing complex constraints. However, when the exact solution lies on the boundary of $\Theta$, solving a constrained problem using the projection process in each of the possibly huge number of iterations will consume much computational time.

Another frequently applied approach is to introduce penalty functions to SA algorithms (see e.g. \cite[Chap.~5 and Chap.~6]{kushner2012stochastic}). By properly constructing a penalty term related to all constraints and adding it to the loss function, the constrained problem is converted to an equivalent unconstrained problem, which eliminates the need to explicitly consider or maintain feasibility during iterations and greatly reduces the computational complexity for each iteration. For instance, \cite{wang_stochastic_2008} constructed modified loss functions with several different kinds of penalty terms and solved these unconstrained problems by the SPSA algorithm. However, the penalty function approach is sensitive to the choice of weighting for the penalty and to the functional form of the penalty. Further, the penalty-based method generally fails to guarantee feasibility for any practical solution from a finite number of iterations. Generally, although implementation of the penalty approach avoids focusing on feasibility in each iteration, the challenges in implementing the method are significant.

In this paper, we propose a switch updating (SU) algorithm based on the SPSA algorithm to cope with constrained problems of the form in \eqref{problem} under some assumptions different from those used in the constrained methods above. With finite feasibility evaluations on the current point in each iteration, our algorithm chooses either an SPSA step based on two measurements of the objective loss function $L(\uptheta)$ or a gradient descent step based on one of the non-satisfied explicit constraint functions $q_i(\uptheta)$. The SU method leads to an updating formula with changes to $\uptheta$ guided explicitly by the loss function or the constraints. Therefore, this SPSA-based SU algorithm directly measures or computes the gradients during iterations, eliminating extra hyper-parameter selection without dealing with any time-consuming process to attain feasibility, such as needing to make projections to the feasible set $\Theta$. Further, the algorithm ensures feasibility for the final solution. This appears to be a novel improvement of SPSA in the constrained stochastic optimization setting.

The organization of the rest of the paper is as follows. Section 2 proposes the SPSA-based switch updating algorithm. Section 3 studies its convergence. Section 4 gives its MSE convergence rate. Section 5 analyzes the asymptotic proportion of iterations where loss function measurements are needed in order to justify that the feasible region can be reached infinitely often. Section 6 presents numerical experiments, comparing the performance of the SU algorithm with the SPSA algorithm using different kinds of penalty ideas. Finally, Section 7 offers closing remarks.

\section{SPSA-based Switch Updating Algorithm}




The basic idea in the switch updating method is that, as long as the current value of $\uptheta$ is infeasible, we choose the gradients of the infeasible constraints and perform the deterministic first-order gradient descent steps to pull the point back to the feasible region. Such a switch updating idea was proposed in \cite{articlepolyak} for dealing with only one constraint function (i.e. $m=1$) in a deterministic setting. The setting of \cite{articlepolyak} sheds light on solving constrained stochastic optimization problems, where we adopt the similar approach to deal with the constraints in efficient stochastic approximation (SA) methods.

The SPSA algorithm analyzed in~\cite[Chap.~7]{spall2005introduction} and ~\cite{spall1998overview} is based on the use of noisy measurements of $L(\uptheta)$, say $y(\uptheta)=L(\uptheta)+\upvarepsilon(\uptheta)$, to approximate the gradient of $L$ for any valid $\uptheta$ without the need for $L(\uptheta)$ or its gradient $g(\uptheta)$. It only uses two measurements $y(\uptheta)$ to update the estimate of $\uptheta$ via an efficient gradient estimate in each step. To combine the classic SPSA algorithm with the idea of the switch updating method at each iteration $k$, we add a switch updating step to keep the $k^{th}$ iteration value $\hat{\uptheta}_k$ feasible after performing the SPSA step in each iteration. In other words, when the current point is feasible, it runs one step based on the SPSA algorithm, which will cost two loss function measurements; otherwise, it runs several steps based on the gradients of the infeasible constraints until the feasibility is attained. We propose the SU algorithm for solving problem \eqref{problem} below. We present the algorithm in two parts. The first, Algorithm \ref{SU}, is the driver that updates the estimate of $\uptheta$ based on new measurements $y(\uptheta)$. Algorithm \ref{SU} calls the second part, Algorithm \ref{reselection}, as needed, to ensure that the constraints are met. Once constraints are met, the estimate for $\uptheta$ is returned to Algorithm \ref{SU} for updating with new function measurements.

In the SU algorithm, $k$ is the iteration counter up to a total number of iterations $K$ where we need to measure $L(\uptheta)$ in Algorithm \ref{SU}, and we use Algorithm \ref{reselection} to meet the constraints. We sample $\Delta_k \in \mathbb{R}^p$ based on some distribution; a commonly adopted one is the $\pm1$ Bernoulli distribution for each of the components of $\Delta_k$ as in \cite{sadegh1998optimal}. 

\begin{algorithm}
\renewcommand{\thealgorithm}{1(a)}
\caption{The SPSA-based Switch Updating Algorithm}
\label{SU}
\begin{algorithmic}
\REQUIRE{an initial point $\hat{\uptheta}_0$, $\Tilde{\uptheta}_0=\hat{\uptheta}_0$, $a>0$, $A \geq 0$, $c>0$, $\upalpha \in (0.5, 1]$, $\upgamma \in [\upalpha/6, \upalpha-0.5)$, $\upbeta \geq 0$, $k=0$, $K>0$ (if known), the switch updating model $\mathcal{M}$.}
\STATE{$a_0=a/(1+A)^\upalpha, \hat{\uptheta}_0=\mathcal{M}(\hat{\uptheta}_0, 0, \upbeta)$}
\WHILE{$k \leq K$}
\STATE{$a_k=a/(k+1+A)^\upalpha$, $c_k = c/(k+1)^\upgamma$}
\STATE{$\hat{g}_k(\hat{\uptheta}_k)=[y(\hat{\uptheta}_k+c_k\Delta_k)-y(\hat{\uptheta}_k-c_k\Delta_k)]/(2c_k\Delta_k)$, where $1/\Delta_{k} = [1/\Delta_{k1},...,1/\Delta_{kp}]^T$ \\ 
$\hat{\uptheta}_{k+1} = \hat{\uptheta}_k - a_k\hat{g}_k(\hat{\uptheta}_k)$, $\hat{\uptheta}_{k+1}=\mathcal{M}(\hat{\uptheta}_{k+1}, k, \upbeta)$, $k=k+1$}
\ENDWHILE
\ENSURE{$\hat{\uptheta}_K \in \Theta$.}
\end{algorithmic}
\end{algorithm}

\begin{algorithm}
\renewcommand{\thealgorithm}{1(b)}
\caption{The Switch Updating Model $\mathcal{M}$}
\label{reselection}
\begin{algorithmic}
\REQUIRE{$\hat{\uptheta}$, $k$ and $\upbeta$.}
\STATE{set $l=0$}
\WHILE{$\hat{\uptheta} \notin \Theta$}
\FOR{$i = 1,2,...,m$}
\IF{$q_i(\hat{\uptheta}) > 0$}
\STATE{$a'_{l}=a'_{l}(k, \upbeta)$, $\hat{\uptheta} = \hat{\uptheta} - a'_{l}\nabla q_i(\hat{\uptheta})$, $l=l+1$ \\
\textbf{break}}
\ENDIF
\ENDFOR
\ENDWHILE
\ENSURE{$\hat{\uptheta} \in \Theta$.}
\end{algorithmic}
\end{algorithm}



The switch updating model $\mathcal{M}$ outputs feasible iterates by modifying $\hat{\uptheta}_k$ at each $k$, as needed in order to meet the constraints. The modification is carried out by a gradient descent method with iteration counter $l$ and a step size $a'_{l}$. As an instance, we pick $a'_{l}(k,\upbeta)=a_k(k+l+1)^{\upbeta}/(k+2l+1)^{\upbeta}$. In this model, $a'_{l}$ is a non-increasing sequence with an adjustable lower bound $a_k/2^\upbeta$. This lower bound tends to zero as we adjust $\upbeta \to \infty$. It is a special case that $a'_{l}=a_k$ when $\upbeta=0$, which is a simple and convenient choice to be applied in most situations. This corresponds to a ``stair-step'' decay of a reduced step size at each entry into Algorithm \ref{reselection}, but a constant step size within Algorithm \ref{reselection}. But in some settings, we need $\upbeta>0$ to make sure $\hat{\uptheta}_{k} \in \Theta$ after finite number of iterations.

For each step of the SU algorithm, the SPSA steps based on $y(\uptheta)$ contribute to attaining descent, while the gradient steps based on $q_i(\uptheta)$ are used to ensure feasibility. The combination of them shown in the SU algorithm provides convergence results to the exact solution of problem \eqref{problem} with feasibility guaranteed as the number of steps increases.

In the following sections of analysis, we set $t$ as the counter of the total number of gradient-based steps in Algorithm \ref{SU}, which includes both $\hat{g}_k$ and $\nabla q_i, i=1,...,m$, and let $\Tilde{a}_t$ represent the step-sizes, $\Tilde{\uptheta}_t$ represent the iterates before each gradient-based step accordingly. For every single replication of the stochastic approximation process, there exists a multi-to-one relationship between $t$ and $k$. Note that there might be several values of $t$ corresponding to each value of $k$, which is because the updates in Algorithm \ref{reselection} only cause increments of $t$, but do not cause increments of $k$.

\section{Convergence Analysis}


For notation simplification, we first define $q(\uptheta)$ as: 

\begin{align}
    q(\uptheta)\equiv
    \begin{cases}
        q_1(\uptheta), &{q_1(\uptheta)>0}, \\
        q_2(\uptheta), &{q_1(\uptheta) \leq 0, q_2(\uptheta)>0}, \\
        q_3(\uptheta), &{q_1(\uptheta),q_2(\uptheta) \leq 0, q_3(\uptheta)>0}, \\
        \cdots\cdots \\
        q_m(\uptheta), &{q_1(\uptheta),...,q_{m-1}(\uptheta) \leq 0, q_m(\uptheta)>0}.
    \end{cases}
\end{align}

Then we define a function $\mathcal{L}(\uptheta)$ as:
\begin{align}
    \mathcal{L}(\uptheta) \equiv
    \begin{cases}
        q(\uptheta), & {\uptheta \notin \Theta}, \\
        L(\uptheta), &{\uptheta \in \Theta}.
    \end{cases}
\end{align}

Algorithm 1(a) uses Algorithm 1(b) as the switch updating model $\mathcal{M}$ and allows the transition from $k$ to $k+1$. Recall that we gave a relationship between $t$ and $k$ in each replication in Section 2, and here we use $\Tilde{\uptheta}_t$ to represent the points after each gradient-based step. The updating formula can be written as:
\begin{align}
    \Tilde{\uptheta}_{t+1}=\Tilde{\uptheta}_t-\Tilde{a}_t\hat{h}_t(\Tilde{\uptheta}_t),
    \label{updatingformula1}
\end{align}

where
\begin{align}
    {\hat{h}_t(\Tilde{\uptheta}_t)=
    \begin{cases}
        \nabla q(\Tilde{\uptheta}_t), & {\Tilde{\uptheta}_t \notin \Theta}, \\
        \hat{g}_{k}(\Tilde{\uptheta}_t), &{\Tilde{\uptheta}_t \in \Theta},
    \end{cases}}
\end{align}


and
\begin{align}
    &\hat{g}_{k}(\Tilde{\uptheta}_t)=\frac{[y(\Tilde{\uptheta}^+_t)-y(\Tilde{\uptheta}^-_t)]}{2c_{k}\Delta_{k}}, \nonumber\\
    &\Tilde{\uptheta}^+_t=\Tilde{\uptheta}_t+c_{k}\Delta_{k}, \Tilde{\uptheta}^-_t=\Tilde{\uptheta}_t-c_{k}\Delta_{k}.
\end{align}

The algorithm representation in \eqref{updatingformula1} is implementable since $\hat{h}(\uptheta_t)$ is computable. For purposes of theoretical analysis, we may decompose \eqref{updatingformula1} as follows:
\begin{align}
    \Tilde{\uptheta}_{t+1}=\Tilde{\uptheta}_t-\Tilde{a}_th(\Tilde{\uptheta}_t)-\Tilde{a}_tb_t-\Tilde{a}_te_t,
    \label{updatingformula2}
\end{align}

where $b_t$ represents the bias in $\hat{h}_t(\Tilde{\uptheta}_t)$, and $e_t$ denotes the noise term. Let $\Im_t=\{\Tilde{\uptheta}_0,...,\Tilde{\uptheta}_t,\Delta_0,...\Delta_{k-1}\}$, $g(\uptheta)=\nabla L(\uptheta)$, $\upvarepsilon_{k}^+=\upvarepsilon(\Tilde{\uptheta}_t^{+})$ and $\upvarepsilon_{k}^-=\upvarepsilon(\Tilde{\uptheta}_t^-)$. We show each term in \eqref{updatingformula2} as:
\begin{align}
    h(\Tilde{\uptheta}_t)=
    \begin{cases}
        \nabla q(\Tilde{\uptheta}_t), & {\Tilde{\uptheta}_t \notin \Theta}, \\
        g(\Tilde{\uptheta}_t), &{\Tilde{\uptheta}_t \in \Theta},
    \end{cases}
\end{align}

\begin{align}
    b_t&=\mathbb{E}[\hat{h}_t(\Tilde{\uptheta}_t)|\Im_t]-h(\Tilde{\uptheta}_t) \nonumber\\
    &=
    \begin{cases}
        \Bar{g}_{k}(\Tilde{\uptheta}_t)-g(\Tilde{\uptheta}_t), &{q_i(\Tilde{\uptheta}_t) \leq 0, i=1,2,...,m}, \\
        0, &{\text{otherwise}},
    \end{cases}
\end{align}

\begin{align}
    e_t&=\hat{h}_t(\Tilde{\uptheta}_t)-\mathbb{E}[\hat{h}_t(\Tilde{\uptheta}_t)|\Im_t] \nonumber\\
    &=
    \begin{cases}
        \hat{g}_{k}(\Tilde{\uptheta}_t)-\Bar{g}_{k}(\Tilde{\uptheta}_t), &{q_i(\Tilde{\uptheta}_t)\leq0, i=1,2,...,m}, \\
        0, &{\text{otherwise}},
    \end{cases}
\end{align}

where
\begin{align}
    \Bar{g}_{k}(\Tilde{\uptheta}_t)=\mathbb{E}[\hat{g}_{k}(\Tilde{\uptheta}_t)|\Im_t]=\mathbb{E}\left[\frac{[L(\Tilde{\uptheta}^+_t)-L(\Tilde{\uptheta}^-_t)]}{2c_{k}\Delta_{k}}\Bigg |\Im_t\right].
\end{align}



For convenience, we introduce $k=\upkappa(t)$ to show the relationship of $t$ and $k$. Based on the convergence result for the SPSA algorithm analyzed in \cite{spall1992multivariate}, and according to the updating formula \eqref{updatingformula2}, we make the following assumptions to ensure almost sure convergence for the SU algorithm:
\begin{assumption}
    $a_k > 0$, $c_k > 0$, $a_k \to 0$, $c_k \to 0$, $\sum_{k=0}^{\infty}a_k = \infty$, $\sum_{k=0}^{\infty}a_k^2/c_k^2 < \infty$. $\forall t$, $\exists t'<\infty \ \mathrm{s.t.} \ \upkappa(t+t')>\upkappa(t)$.
    \label{assumption1'}
\end{assumption}


\begin{assumption}
    $\forall k, t, j$, $\Delta_{kj}$ are i.i.d. and symmetrically distributed about $0$ and $\exists \updelta,\upalpha_0,\upalpha_1,\upalpha_2>0$ $s.t. \mathbb{E}[|\upvarepsilon_k^{\pm}|^{2+\updelta}] \leq \upalpha_0$, $\mathbb{E}[|L(\Tilde{\uptheta}_t^{\pm})|^{2+\updelta}|\Tilde{\uptheta}_t \in \Theta] \leq \upalpha_1$, $|\Delta_{kj}| \leq \upalpha_2$, and $\mathbb{E}[|\Delta_{kj}|^{-2-\updelta}] \leq \upalpha_3$.
    \label{assumption2'}
\end{assumption}

\begin{assumption}
    $\forall t, \|\Tilde{\uptheta}_t\|<\infty$ almost surely.
    \label{assumption3'}
\end{assumption}

\begin{assumption}
    $\uptheta^*$ is an asymptotically stable solution of the differential equation $dx(s)/ds=-h(x)$.
    \label{assumption4'}
\end{assumption}

\begin{assumption}
    Let $D(\uptheta^*)=\{x_0|\lim_{s \to \infty}x(s|x_0)=\uptheta^*\}$ where $x(s|x_0)$ denotes the solution to the differential equation $dx(s)/ds=-h(x)$ based on initial conditions $x_0$. Then there exists a compact set $S \subseteq D(\uptheta^*)$ such that $\Tilde{\uptheta}_t \in S$ infinitely often for almost all sample points.
    \label{assumption5'}
\end{assumption}

For Assumption \ref{assumption1'}, since the gain sequence related to the gradients of the constraint functions does not decay to $0$ before reaching the next feasible point, it is required that their values should be small enough to ensure attaining feasibility in finite number of iterations. This is equivalent that $a_0$ is not extremely large and $\upbeta$ is sufficiently large to avoid diverging, which is not difficult to be satisfied. Assumption \ref{assumption4'} constructs the ODE equation based on not only the gradient of the objective loss function but also the gradients of all the constraint functions. Although $h(\uptheta)$ defined for the SU algorithm is not a continuous function on $\mathbb{R}^p$, $\uptheta^*$ can still be an asymptotically stable solution for the ODE shown in Assumption \ref{assumption4'}. Moreover, $g(\uptheta^*)$, $\nabla q_i(\uptheta^*)$ $(i=1,...,m)$ being nonzero can make $\uptheta^*$ even more stable than under condition $g(\uptheta^*)=0$ in unconstrained problem cases. Assumption \ref{assumption4'} is crucial to Proposition \ref{proposition1}, and as a special case to satisfy it, it is similar to requiring the objective loss function and all the constraint functions to be partly strictly convex.

Proposition \ref{proposition1} and Corollary \ref{corollary1} given below show almost sure convergence for the SU algorithm. 
\begin{proposition}
    Suppose that Assumptions \ref{assumption1'}--\ref{assumption5'} hold. Then when $t \to \infty$, $\Tilde{\uptheta}_t \to \uptheta^*$ almost surely in the SU algorithm.
    \label{proposition1}
\end{proposition}
\begin{proof}
    

    According to Assumption \ref{assumption1'}, \ref{assumption2'} and Lemma 1 in \cite{spall1992multivariate}, we can directly know:
    \begin{align}
        \|b_t\|<\infty \ \forall t \ \text{and} \ b_t \to 0 \ \text{almost surely}.
        \label{prop1assum1}
    \end{align}

    In as much as $\{\sum_{i=t}^n\Tilde{a}_ie_i\}_{n \geq l}$ is a martingale sequence, and $\mathbb{E}[e_i^Te_j]=\mathbb{E}[e_i^T\mathbb{E}[e_j|\Im_j]] \ \forall i<j$, we have:
    \begin{align}
        P(\sup_{n \geq t}\|\sum_{i=t}^n\Tilde{a}_ie_i\| \geq \eta) &\leq \eta^{-2}\mathbb{E}[\|\sum_{i=t}^n\Tilde{a}_ie_i\|^2] \nonumber\\
        &=\eta^{-2}\sum_{i=t}^n\Tilde{a}_i^2\mathbb{E}[\|e_i^2\|].
        \label{prop1e1}
    \end{align}
    
    
    To establish an upper bound for $\mathbb{E}\|e_t^2\|$, as $e_t=0$ when $\Tilde{\uptheta}_t$ violates one or more constraints, we only need to consider the case that $\Tilde{\uptheta}_t$ is feasible. According to H$\ddot{\text{o}}$lder's theorem and Assumption \ref{assumption2'}, we can obtain $\mathbb{E}[|\upvarepsilon_{k}^{\pm}|^2] \leq \upalpha_0^{\frac{1}{2+\updelta}}=\upalpha_0'$. By the similar process, we can find that $\mathbb{E}[L(\Tilde{\uptheta}_t^{\pm})^2|\Tilde{\uptheta}_t \in \Theta]$ and $\mathbb{E}[\Delta_{{k}j}^{-2}]$ are also both bounded. Let the upper bound for them be $\upalpha_1'$ and $\upalpha_3'$, then we have:
    \begin{align}
        \mathbb{E}[\hat{g}_{{k}j}(\Tilde{\uptheta}_t)^2] \leq &\frac{1}{4}\mathbb{E}[(L(\Tilde{\uptheta}^+_t)-L(\Tilde{\uptheta}^-_t)+\upvarepsilon_{k}^{+}-\upvarepsilon_{k}^{-})^2] \nonumber\\
        &\mathbb{E}[(c_{k}\Delta_{{k}j})^{-2}] \nonumber\\
        \leq &2(\upalpha_1'+\upalpha_0')\upalpha_3'c_{k}^{-2}.
    \end{align}
    Thus, when $\Tilde{\uptheta}_t$ is feasible, for $\mathbb{E}[\|e_t\|^2]$ we have:
    \begin{align}
        \mathbb{E}[\|e_t\|^2] &\leq p\max_{1 \leq j \leq p}\mathbb{E}[e_{tj}^2] \nonumber\\
        &=p\max_{1\leq j \leq p}\mathbb{E}[(\hat{g}_{{k}j}(\Tilde{\uptheta}_t)-\mathbb{E}[\hat{g}_{{k}j}(\Tilde{\uptheta}_t)])^2] \nonumber\\
        &=p\max_{1\leq j \leq p}[\mathbb{E}[\hat{g}_{{k}j}(\Tilde{\uptheta}_t)^2]-(\mathbb{E}[\hat{g}_{{k}j}(\Tilde{\uptheta}_t)])^2] \nonumber\\
        &\leq p\max_{1\leq j \leq p}\mathbb{E}[\hat{g}_{{k}j}(\Tilde{\uptheta}_t)^2] \nonumber\\
        &\leq 2p(\upalpha_1'+\upalpha_0')\upalpha_3'c_{k}^{-2}.
        \label{prop1e2}
    \end{align}
    
    By \eqref{prop1e1}, \eqref{prop1e2} and Assumption \ref{assumption1'}, we have:
    \begin{align}
        \lim_{t \to \infty}P(\sup_{n \geq t}\|\sum_{i=t}^n\Tilde{a}_ie_i\| \geq \eta) \leq 2p\eta^{-2}(\upalpha_1'+\upalpha_0')\upalpha_3'\lim_{t \to \infty}\sum_{k=\upkappa(t)}^{\upkappa(n)}a_k^2c_k^{-2}=0.
        \label{prop1assum2}
    \end{align}
    
    Therefore, according to Assumptions \ref{assumption1'}, \ref{assumption3'}, \ref{assumption4'}, \ref{assumption5'}, \eqref{prop1assum1} and \eqref{prop1assum2}, the conditions of Theorem 2.3.1 in \cite[Chap.~2]{kushner2012stochastic} are satisfied. Based on this theorem, we can finally reach the result that when $t \to \infty$, $\Tilde{\uptheta}_t \to \uptheta^*$.
\end{proof}

\begin{corollary}
    Suppose that Assumptions \ref{assumption1'}--\ref{assumption5'} hold. Then when $k \to \infty$, $\hat{\uptheta}_k \to \uptheta^*$ almost surely in the SU algorithm.
    \label{corollary1}
\end{corollary}
\begin{proof}
    According to Assumption \ref{assumption1'}, whenever $\Tilde{\uptheta}_t$ is infeasible in the SU algorithm, we will attain a feasible point for problem \eqref{problem} in finite numbers of iterations. This indicates that in the SU algorithm, $k=\upkappa(t) \to \infty$ as $t \to \infty$. Then according to Proposition \ref{proposition1}, as a subsequence of $\{\Tilde{\uptheta}_t\}$, $\hat{\uptheta}_k \to \uptheta^*$ almost surely in the SU algorithm.
\end{proof}


\section{Convergence Rate Analysis}
This part shows the convergence rate of the SU algorithm. We will give the asymptotic convergence rate performance of our algorithm in the big-$O$ sense. 

We use a mean gradient-like expression $\bar{h}(\Tilde{\uptheta}_t)=\mathbb{E}[\hat{h}_t(\Tilde{\uptheta}_t)|\Im_t]$ for our analysis in this part, which is:
\begin{align}
    \bar{h}(\Tilde{\uptheta}_t)=
    \begin{cases}
        \nabla q(\Tilde{\uptheta}_t), & {\Tilde{\uptheta}_t \notin \Theta}, \\
        \bar{g}_{k}(\Tilde{\uptheta}_t), &{\Tilde{\uptheta}_t \in \Theta}.
    \end{cases}
    \label{bar}
\end{align}
As an example, if we assume that each element of $\Delta_{k} \in \mathbb{R}^{p}$ here follows $\pm 1$ Bernoulli distribution and $\sum_{\Delta_{k}}$ is the summation of all the possible $\Delta_{k}$, then we have:
\begin{align}
    \bar{g}_{k}(\Tilde{\uptheta}_t)=\frac{1}{2^p}\sum\nolimits_{\Delta_{k}}\frac{L(\Tilde{\uptheta}^{+}_t)-L(\Tilde{\uptheta}^{-}_t)}{2c_{k}}\Delta_{k}^{-1}.
\end{align}

Then we make the following assumptions to ensure the MSE convergence rate of the SU algorithm: 

\begin{assumption}
    The components of $\Delta_{k}$ are i.i.d. random variables and $0<\Delta_{k}^{-T}\Delta_{k}^{-1}\leq d<\infty$.
    \label{assumption1''}
\end{assumption}

\begin{assumption}
    For all $t$, $\mathbb{E}[\upvarepsilon_{k}^{+}-\upvarepsilon_{k}^{-}|\Im_t,\Delta_{k}] = 0$ almost surely, and $\mathrm{var}(\upvarepsilon_{k}^{\pm})$ is uniformly bounded in ${k}$.
    \label{assumption2''}
\end{assumption}

\begin{assumption}
    $\mathbb{E}[(L(\Tilde{\uptheta}^{+}_{t})-L(\Tilde{\uptheta}^{-}_{t}))^2|\Im_t]$ and $\|\nabla q_i(\Tilde{\uptheta}_t)^2\|$ are uniformly bounded for all $t$ almost surely.
    \label{assumption3''}
\end{assumption}

\begin{assumption}
    $\exists M>0$ such that when $\Tilde{\uptheta}_t$ is infeasible, it will not need more than $M$ iterations to get into the feasible set $\Theta$. 
    \label{assumption4''}
\end{assumption}

\begin{assumption}
    There exists $\upmu > 0$ such that $(\Tilde{\uptheta}_t-\uptheta^{*})^{T}\bar{h}(\Tilde{\uptheta}_t)-\upmu(\Tilde{\uptheta}_t-\uptheta^{*})^{T}(\Tilde{\uptheta}_t-\uptheta^{*})\geq 0$ for all $t$.
    \label{assumption5''}
\end{assumption}

Let us comment on the above conditions. For Assumption \ref{assumption3''}, it means that the sequence of \{$\Tilde{\uptheta}_t$\} should not jump to the ``far away'' area so often, which is a natural consequence if $\|\Tilde{\uptheta}_t\|$ is almost surely bounded. For Assumption \ref{assumption4''}, it is easy to satisfy if, to avoid diverging, we select a value of $a_0$ that is not too large. For Assumption \ref{assumption5''}, it is similar to state that $\mathcal{L}(\uptheta)$ is quasi-strongly convex if $\Bar{g}_{\kappa(t)}(\Tilde{\uptheta}_t)=g(\Tilde{\uptheta}_t) \ \forall t$ in some cases. We say $f(\uptheta)$ is a $\upmu'$-quasi-strongly convex function if with $\uptheta^*$ as its unique optimal solution, there exists $\upmu'>0$ such that for any $\uptheta$:
\begin{align}
    f(\uptheta^*)\geq f(\uptheta)+\nabla f(\uptheta)^T(\uptheta^*-\uptheta)+\frac{\upmu'}{2}\|\uptheta-\uptheta^*\|^2.
    \label{quasi}
\end{align}

Specifically, since $\bar{h}(\Tilde{\uptheta}_t)$ is composed of several parts as shown in \eqref{bar}: $\nabla q_i(\Tilde{\uptheta}_t)$ $(i=1,...,m)$ and $\bar{g}_{k}(\Tilde{\uptheta}_t)$, Assumption \ref{assumption5''} actually requires all the constraint functions and the loss function to be partly quasi-strongly convex, with $\uptheta^*$ being the same optimal solution for them.

Proposition \ref{upperbound} given below shows the MSE convergence rate of the SU algorithm.


%
%


\begin{proposition}
    Suppose that Assumptions \ref{assumption1''}--\ref{assumption5''} hold. Then the MSE convergence rate of the SU algorithm is $\mathbb{E}[\|\hat{\uptheta}_{t}-\uptheta^{*}\Vert^{2}]=O(1/t^\upalpha)$. 
    \label{upperbound}
\end{proposition}

\begin{proof}
    First we consider the situation that $\Tilde{\uptheta}_t$ is infeasible. In this case $\Tilde{a}_{t} = \frac{a_k(t+1)^\upbeta}{(t+l+1)^\upbeta}$ and $\hat{h}_t(\Tilde{\uptheta}_t)=\nabla q_i(\Tilde{\uptheta}_t)$ for some $i=1,...,m$. Then we have:
    \begin{align}
        \Tilde{\uptheta}_{t+1} = \Tilde{\uptheta}_{t}-\frac{a_k(t+1)^\upbeta}{(t+l+1)^\upbeta}\nabla q_i(\Tilde{\uptheta}_t).
    \end{align}
    
    Subtracting $\uptheta^{*}$ from both sides and calculating the norm squared, we get:
    \begin{align}
        \|\Tilde{\uptheta}_{t+1}-\uptheta^{*}\|^{2} = &\|\Tilde{\uptheta}_{t}-\uptheta^{*}\|^2 -2\Tilde{a}_t(\Tilde{\uptheta}_t-\uptheta^{*})^{T}\nabla q_i(\Tilde{\uptheta}_t) \nonumber\\
        &+\Tilde{a}_{t}^{2}\|\nabla q_i(\Tilde{\uptheta}_t)\|^2.
    \end{align}
    
    Because $\|\nabla q_i(\Tilde{\uptheta}_t)\|^2$ is uniformly bounded, according to Assumption \ref{assumption3''}, there exists $B > 0$ such that $\|\nabla q_i(\Tilde{\uptheta}_t)\|^2 \leq B$ almost surely. Then we have:
    \begin{align}
        \|\Tilde{\uptheta}_{t+1}-\uptheta^{*}\|^{2} = &\|\Tilde{\uptheta}_{t}-\uptheta^{*}\|^2 -2\Tilde{a}_{t}(\Tilde{\uptheta}_t-\uptheta^{*})^{T}\nabla q_i(\Tilde{\uptheta}_t) \nonumber\\
        &+ \Tilde{a}_{t}^{2}\|\nabla q_i(\Tilde{\uptheta}_t)\|^2 \nonumber\\
        \leq &(1-2\upmu\Tilde{a}_{t})\|\Tilde{\uptheta}_{t}-\uptheta^{*}\|^2+\Tilde{a}_{t}^{2}B \nonumber\\
        = &\left[1-\frac{2\upmu a_k(t+1)^\upbeta}{(t+l+1)^\upbeta}\right]\|\Tilde{\uptheta}_{t}-\uptheta^{*}\|^2 \nonumber\\
        &+\frac{a_k^{2}(t+1)^{2\upbeta}}{(t+l+1)^{2\upbeta}}B \nonumber\\
        \leq &\left(1-\frac{2\upmu a_k}{2^{\upbeta}}\right)\|\Tilde{\uptheta}_{t}-\uptheta^{*}\|^2+a_k^2B.
        \label{inequality1}
    \end{align}
    
    Similarly, when $\Tilde{\uptheta}_{t}$ is feasible with $k=\upkappa(t)$, now $\Tilde{a}_t=a_{\upkappa(t)}=a_k$, and we have:
    \begin{align}
        \Tilde{\uptheta}_{t+1} = \Tilde{\uptheta}_{t}-\Tilde{a}_{t}\hat{g}_{k}(\Tilde{\uptheta}_t)=\Tilde{\uptheta}_{t}-a_k\hat{g}_{k}(\Tilde{\uptheta}_t).
    \end{align}

    Since $\bar{h}(\Tilde{\uptheta}_t)=\mathbb{E}[\hat{h}_t(\Tilde{\uptheta}_t)]$, we still subtract $\uptheta^*$ and calculate the norm squared, and take the expectation condition on $\Tilde{\uptheta}_t$. Then we get the inequality similar as above:
    \begin{align}
        \mathbb{E}[\|\hat{\uptheta}_{t+1}-\uptheta^{*}\|^{2}|\Im_t]=&\|\Tilde{\uptheta}_{t}-\uptheta^{*}\|^2+a_k^2\mathbb{E}[\|\hat{g}_{k}(\Tilde{\uptheta}_t)\|^2|\Im_t] \nonumber\\
        &-2a_k(\Tilde{\uptheta}_t-\uptheta^{*})^{T}\mathbb{E}[\hat{g}_{k}(\Tilde{\uptheta}_t)|\Im_t] \nonumber\\
        \leq&(1-2\upmu a_{k})\|\Tilde{\uptheta}_{t}-\uptheta^{*}\|^2 \nonumber\\
        &+a_{k}^{2}\mathbb{E}[\|\hat{g}_{k}(\Tilde{\uptheta}_t)\|^2|\Im_t].
        \label{lupdate}
    \end{align}
    
    According to Assumption \ref{assumption2''} and \ref{assumption3''}, $\mathbb{E}[(L(\Tilde{\uptheta}^{+}_t)-L(\Tilde{\uptheta}^{-}_t))^2|\Im_t]+\mathbb{E}[(\upvarepsilon_{k}^{+}-\upvarepsilon_{k}^{-})^2|\Im_t]$ is upper bounded. Let us assume that its upper bound is $b$, therefore we have:
    \begin{align}
        \mathbb{E}[\|\hat{g}_{k}(\Tilde{\uptheta}_t)\|^2|\Im_t]=&\mathbb{E}[(L(\Tilde{\uptheta}^{+}_t)-L(\Tilde{\uptheta}^{-}_t))^2\Delta_{k}^{-T}\Delta_{k}^{-1}|\Im_t] \nonumber\\ &+\mathbb{E}[(\upvarepsilon_{k}^{+}-\upvarepsilon_{k}^{-})^2\Delta_{k}^{-T}\Delta_{k}^{-1}|\Im_t] \nonumber\\
        \leq &pd(\mathbb{E}[(L(\Tilde{\uptheta}^{+}_t)-L(\Tilde{\uptheta}^{-}_t))^2|\Im_t] \nonumber\\
        &+\mathbb{E}[(\upvarepsilon_{k}^{+}-\upvarepsilon_{k}^{-})^2|\Im_t]) \nonumber\\
        \leq &pdb.
        \label{gsquare}
    \end{align}
    
    Take \eqref{gsquare} into \eqref{lupdate} and take the expectation, we have:
    \begin{align}
        \mathbb{E}[\|\Tilde{\uptheta}_{t+1}-\uptheta^{*}\|^{2}]\leq(1-2\upmu a_{k})\mathbb{E}[\|\Tilde{\uptheta}_{t}-\uptheta^{*}\|^2]+a_k^2pdb.
        \label{inequality2}
    \end{align}
    
    According to Assumption \ref{assumption4''}, we have $k\leq t\leq Mk$. Considering the inequalities \eqref{inequality1} and \eqref{inequality2}, we let $\upmu^{*} = \min\{\upmu/2^\upbeta,1/(4a_0)\}$ when $0.5<\upalpha<1$, $\upmu^{*} \leq \upmu/2^\upbeta$ (and make sure $2\upmu^*a \neq 1$) when $\upalpha=1$, and $B^{*} = \max\{B,pdb\}$. For any iteration $t$ in the whole process when $0.5<\upalpha<1$, we have $1-2\upmu^*a_t>0$. Then in general:
    \begin{align}
        \mathbb{E}[\|\Tilde{\uptheta}_{t+1}-\uptheta^{*}\|^{2}] &\leq (1-2\upmu^{*}a_k)\mathbb{E}[\|\Tilde{\uptheta}_{t}-\uptheta^{*}\|^2]+a_k^{2}B^* \nonumber\\
        &\leq (1-2\upmu^{*}a_t)\mathbb{E}[\|\Tilde{\uptheta}_{t}-\uptheta^{*}\|^2]+a_t^{2}M^{\upalpha}B^*.
    \end{align}
    
    Using the proof in \cite{wang2013rate}, we get the following inequality:
    \begin{align}
        \mathbb{E}[\|\Tilde{\uptheta}_{t}-\uptheta^{*}\|^{2}] \leq 
        \begin{cases}
        \exp\left\{\frac{2\upmu^*a[(1+A)^{1-\upalpha}-(1+A+t)^{1-\upalpha}]}{1-\upalpha}\right\} \left[\|\Tilde{\uptheta}_0-\uptheta^{*}\|^2-\frac{T(t,\upalpha)}{(1+A)^\upalpha}\right]+\frac{T(t,\upalpha)}{(1+A+t)^\upalpha},& 0.5<\upalpha<1,\\
        \frac{(1+A)^{2\upmu^*a}}{(1+A+t)^{2\upmu^*a}}\left[\|\Tilde{\uptheta}_0-\uptheta^*\|^2-\frac{T(t,1)}{1+A}\right]+\frac{T(t,1)}{1+A+t}, & \upalpha = 1,
        \end{cases}
        \label{convergencerate}
    \end{align}
    
    where 
    \begin{align}
        &T(t,\upalpha) = \frac{M^{\upalpha}B^{*}a^2C(\upalpha)}{2\upmu^*a-\upalpha/[1+A+f(t,\upalpha)]^{1-\upalpha}}, T(t,1)= \frac{MB^{*}a^2C(1)}{2\upmu^*a-1}, \nonumber\\
        &C(\upalpha) = \exp\left\{\frac{2\upmu^*a}{1-\upalpha}\left[(2+A)^{1-\upalpha}-(1+A)^{1-\upalpha}\right]\right\} \left(\frac{2+A}{1+A}\right)^{2\upalpha},C(1)=(\frac{2+A}{1+A})^{2\upmu^*a+2},\nonumber\\
        &f(t,\upalpha) = \left\{\frac{\int_0^t(1+A+x)^{-2\upalpha}\exp\{\frac{2\upmu^*a(1+A+x)^{1-\upalpha}}{1-\upalpha}\}dx}{\int_0^t(1+A+x)^{-1-\upalpha}\exp\{\frac{2\upmu^*a(1+A+x)^{1-\upalpha}}{1-\upalpha}\}dx}\right\}^{\frac{1}{1-\upalpha}}-1-A.
        \label{35}
    \end{align}

    For the case $0.5<\upalpha<1$, we can rewrite the equation of $f(t,\upalpha)$ as:
    \begin{align}
        &\frac{1}{[1+A+f(t,\upalpha)]^{1-\upalpha}}=\frac{\int_0^t(1+A+x)^{-1-\upalpha}\exp\{\frac{2\upmu^*a(1+A+x)^{1-\upalpha}}{1-\upalpha}\}dx}{\int_0^t(1+A+x)^{-2\upalpha}\exp\{\frac{2\upmu^*a(1+A+x)^{1-\upalpha}}{1-\upalpha}\}dx}.
    \end{align}
    
    By L'H$\hat{\text{o}}$pital's rule, when $t \to \infty$, we have:
    \begin{align}
        &\lim_{t \to \infty} \frac{1}{[1+A+f(t,\upalpha)]^{1-\upalpha}} \nonumber\\
        = &\lim_{t \to \infty} \frac{\int_0^t(1+A+x)^{-1-\upalpha}\exp\{\frac{2\upmu^*a(1+A+x)^{1-\upalpha}}{1-\upalpha}\}dx}{\int_0^t(1+A+x)^{-2\upalpha}\exp\{\frac{2\upmu^*a(1+A+x)^{1-\upalpha}}{1-\upalpha}\}dx} \nonumber\\
        = &\lim_{t \to \infty} (1+A+t)^{\upalpha-1} \nonumber\\
        = &0.
        \label{37}
    \end{align}
    
    From \eqref{37}, when $0.5<\upalpha<1$, we know that as $t \to \infty$, $f(t,\upalpha) \to \infty$, then we can get $T(t,\upalpha) \to \frac{M^{\upalpha}B^{*}aC(\upalpha)}{2\upmu^*}$, which is a constant. Similarly, when $\upalpha=1$, we have $T(t,1)=\frac{MB^{*}aC(1)}{2\upmu^*a-1}$ which is also a constant. When $t$ is large enough, the effect of $A$ can be ignored. Therefore, for the case $0.5<\upalpha<1$, the convergence rate of the second term in RHS of \eqref{convergencerate} is $O(1/t^\upalpha)$; for the first term in its RHS, we only need to consider the exponential part, and it is obvious that this part goes to $0$ faster than the second term, which indicates the convergence rate is $O(1/t^\upalpha)$. For the case $\upalpha=1$, if $2\mu^{*}a>1$, then we have $T(t,1)>0$, which indicates the convergence rate is $O(1/t^\upalpha)=O(1/t)$; otherwise, if $2\mu^{*}a<1$, $T(t,1)<0$ and the convergence rate is $O(1/t^{2\upmu^{*}a})$. Note that the value of $a$ is adjustable in this case to satisfy $2\mu^{*}a>1$ and achieve the faster convergence rate. Therefore, we can finally get the convergence rate as $\mathbb{E}[\|\hat{\uptheta}_{t}-\uptheta^{*}\Vert^{2}]=O(1/t^\upalpha)$. \\
\end{proof}

\section{Asymptotic Result for Number of Loss Function Measurements}


In this section, we are going to gain some insight into the perspective of the number of loss function measurements. For the SU algorithm, the loss function is not measured in every iteration. Here we give an asymptotic result for the ratio of iterations that necessitates loss function measurements in the SU algorithm. This result also justifies the last part in Assumption \ref{assumption1'} that the feasible region can be reached infinitely often.

According to Karush-Kuhn-Tucker (KKT) conditions, we define the set of KKT points as the asymptotically stable optimal solutions for problem \eqref{problem}.
\begin{definition}
    The set of KKT points are defined as: $\{\uptheta \ | \forall i \ q_i(\uptheta) \leq 0, \exists \uplambda_i \geq 0 \ \mathrm{s.t.} \uplambda_iq_i(\uptheta)=0, \frac{dL(\uptheta)}{d\uptheta} + \sum_{i=1}^m\uplambda_i\frac{dq_i(\uptheta)}{d\uptheta} = 0\}$.
    \label{definitionKKT}
\end{definition}

To further let KKT points necessarily be the optimal solutions for problem \eqref{problem}, some type of constraint qualification should be satisfied. We then define a common type of constraint qualification called linear independence constraint qualification (LICQ) condition.

\begin{definition}
    Let set $A(\uptheta)=\{i|q_i(\uptheta)=0, i=1,...,m\}$. When $\nabla q_i(\uptheta)$ $(i \in A(\uptheta))$ are linearly independent at $\uptheta$, the LICQ condition is satisfied at $\uptheta$ for problem \eqref{problem}.
    \label{definitionLICQ}
\end{definition}

Moreover, if $A(\uptheta)=\phi$, then the LICQ condition is automatically satisfied at $\uptheta$.

\begin{proposition}
    Suppose that Assumptions \ref{assumption1'}--\ref{assumption5'} hold. Let $\uptheta^*$ be a KKT point and the optimal solution for problem \eqref{problem}, and $\uplambda_i$ $(i=1,...,m)$ be the Lagrange multipliers corresponding to $\uptheta^*$ in Definition \ref{definitionKKT}. Assume the LICQ condition is satisfied at $\uptheta^*$ for problem \eqref{problem}. Then when $t \to \infty$, the proportion of the iterations that make loss function measurements in the SU algorithm with $\upalpha \in (0.5, 1)$ is $\frac{1}{1+\sum_{i=1}^m\uplambda_i}$.\label{proportion}
\end{proposition}
\begin{proof}
    We know that the optimal solution $\uptheta^*$ for our problem \eqref{problem} is a KKT point for it. Due to Assumption \ref{assumption1'}, $k=\kappa(t) \to \infty$ as $t \to \infty$. Then according to Proposition \ref{proposition1}, when $t \to \infty$, $\hat{\uptheta}_k \to \uptheta^*$ almost surely in the SU algorithm. Considering $k'=k^{\frac{1+\upalpha}{2}} \to \infty$ iterations after the $t$-th iteration with $k=\upkappa(t)$, we have:
    \begin{align}
        \lim_{t \to \infty}\Tilde{\uptheta}_t=\lim_{t \to \infty}[\Tilde{\uptheta}_t-\sum_{i=0}^{k'-1}\Tilde{a}_{t+i}\hat{h}_{t+i}(\Tilde{\uptheta}_{t+i})].
        \label{tlinitial}
    \end{align}

    Since $b_t=\mathbb{E}[\hat{h}_t(\Tilde{\uptheta}_t)|\Tilde{\uptheta}_t]-h(\Tilde{\uptheta}_t)\to0$ and $e_t=\hat{h}_t(\Tilde{\uptheta}_t)-\mathbb{E}[\hat{h}_t(\Tilde{\uptheta}_t)|\Tilde{\uptheta}_t]\to0$ as $t \to \infty$, we have $\hat{h}_{t+i}(\Tilde{\uptheta}_{t+i}) \to h(\Tilde{\uptheta}_{t+i})$. Note that since $h(\uptheta)$ is not continuous at $\uptheta^*$, we cannot get $h(\Tilde{\uptheta}_{t+i}) \to h(\uptheta^*)$ as $\Tilde{\uptheta}_{t+i} \to \uptheta^*$. Then \eqref{tlinitial} becomes:
    \begin{align}
        \uptheta^*&=\uptheta^*-\lim_{t \to \infty}\sum_{i=0}^{k'-1}\Tilde{a}_{t+i}h(\Tilde{\uptheta}_{t+i}).
    \end{align}
    
    Since $\mathop{\lim}_{t \to \infty}\frac{k'}{t} \leq \mathop{\lim}_{t \to \infty}\frac{k'}{k}=0$, we have $\mathop{\lim}_{t \to \infty}\Tilde{a}_t=\mathop{\lim}_{t \to \infty}\Tilde{a}_{t+i}$ when $0 \leq i<k'$. Then we get:
    \begin{align}
        \lim_{t \to \infty}\sum_{i=0}^{k'-1}\Tilde{a}_{t+i}h(\Tilde{\uptheta}_{t+i})&=0, \nonumber\\
        \lim_{t \to \infty}\Tilde{a}_tk'\frac{\sum_{i=0}^{k'-1}h(\Tilde{\uptheta}_{t+i})}{k'}&=0.
    \end{align}
    
    Since $\lim_{t \to \infty}\Tilde{a}_tk'>\lim_{t \to \infty}\frac{k'}{2^\upbeta}a_{k}k'>0$, we have:
    \begin{align}
        \lim_{t \to \infty}\frac{\sum_{i=0}^{k'-1}h(\Tilde{\uptheta}_{t+i})}{k'}&=0.
    \end{align}
    
    Assume that in $k'$ iterations, we update $k'_0$ times based on functional measurements of $L(\uptheta)$, and $k'_i$ $( i=1,2,...,m)$ times based on $\nabla q_i(\uptheta)$ respectively, then $k'=\sum_{i=0}^mk'_i$. Denote $g(\uptheta)=\nabla L(\uptheta)$, when $t \to \infty$ we have:
    \begin{align}
        \lim_{t \to \infty}\frac{k'_0g(\uptheta^*)+\sum_{i=1}^mk'_i\nabla q_i(\uptheta^*)}{k'}&=0.
        \label{timesmeasurenments}
    \end{align}

    Since the LICQ condition is satisfied at KKT point $\uptheta^*$, then we define a set $A=\{i|q_i(\uptheta^*)=0\}$. For all $i \in A$, $\nabla q_i(\uptheta^*)$ are linearly independent. Therefore, according to Definition \ref{definitionKKT}, there exists a unique set of $\uplambda_i$ $(i=1,2,...,m)$ that $\uplambda_i=0$ if $i \notin A$, and they satisfy:
    \begin{align}
        g(\uptheta^*)+\sum_{i=1}^m\uplambda_i\nabla q_i(\uptheta^*)&=0.
    \end{align}

    Therefore, we finally have:
    \begin{align}
        \lim_{t \to \infty}\frac{k'_0}{k'}=\lim_{t \to \infty}\frac{k'_0}{k'_0+\sum_{i=1}^mk'_i}=\frac{1}{1+\sum_{i=1}^m\uplambda_i}.
    \end{align}

    This result is satisfied when $t \to \infty$, thus it shows the asymptotic proportion of the iterations that make loss function measurements in the SU algorithm. \\
\end{proof}

Proposition \ref{proportion} given above shows that as $t \to \infty$ in the SU algorithm, $k/t$ is a constant value. This also allows the result of Proposition \ref{upperbound} to be equivalently based on the number of loss function measurements.

\section{Numerical Analysis}

\subsection{Overview}
In this section, we implement the SU algorithm solving two different constrained stochastic optimization problems, with quadratic and quartic loss functions respectively. A main difference is that for quadratic loss functions, the bias $b_t=0$ for any $t$, while there usually exists nonzero bias for quartic loss functions during iterations.

We compare our algorithm with the SPSA algorithms based on penalty ideas given in \cite{wang_stochastic_2008}. Basically, by applying penalty ideas, problem \eqref{problem} is constructed as an unconstrained problem in each iteration $k$ as:
\begin{align}
    \min_{\uptheta \in \Theta} \ L_k(\uptheta)=\min \ L(\uptheta)+r_kP(\uptheta),
    \label{penaltyproblem}
\end{align}
where $P:\mathbb{R}^p \to \mathbb{R}^+$ is the penalty function, and the penalty gain sequence $\{r_k\}$ consists of positive values. We consider comparing the SU algorithm with the SPSA algorithms using three different kinds of penalty functions $P(\uptheta)$: the absolute value penalty (AVP) function, the quadratic penalty (QP) function, and the augmented Lagrange (AL) function. 

To initialize the hyper-parameters of these constrained SPSA algorithms, we use the same following gain sequences: $a_k = a/(k + 1 + A)^\upalpha$ and $ c_k = c/(k+1)^\upgamma $. Standard criteria for the selection of $a$, $A$, $c$, $\upalpha$, $\upgamma$ could be seen in \cite[Sect.~7.5]{spall2005introduction}. Furthermore, for the penalty functions, we choose $r_k = r(k+1)^{\uprho}$ as the penalty gain sequence, where $\uprho$ represents the growth rate, and it should satisfy $ \upalpha - \upgamma - 2\uprho > 0 $ and $ 3\upgamma - \upalpha/2 + 3\uprho/2 >0 $ from \cite{wang_stochastic_2008}. For the AVP function, $r_n = r$ is chosen since the minimum of $ L + rP $ is identical to the original problem \eqref{penaltyproblem} for all $r > \Bar{r} $, where $\Bar{r}$ is sure to exist according to Proposition 4.3.1 in \cite{bertsekas1997nonlinear}.

For all the algorithms in our numerical experiments, we use the same values of the random generated noises $\upvarepsilon_k$ and the sample $\Delta_k$ in any given iteration of each replicate to achieve a fair comparison.

\subsection{A Quadratic Example}
The first test case we use is a quadratic example, which comes from \cite{wang_stochastic_2008}:
\begin{align}
    \min\limits_{\uptheta} \ L(\uptheta) &= t_1^2 + t_2^2 + 2t_3^2 + t_4^2 - 5t_1 - 5t_2 - 21t_3 + 7t_4, \nonumber\\
    \mathrm{s.t.} \ q_1(\uptheta) &= 2t_1^2 + t_2^2 + t_3^2 + 2t_1 - t_2 - t_4 - 5 \leq 0, \nonumber\\
    q_2(\uptheta) &= t_1^2 + t_2^2 + t_3^2 + t_4^2 + t_1 - t_2 + t_3 - t_4 - 8 \leq 0, \nonumber\\
    q_3(\uptheta) &= t_1^2 + 2t_2^2 + t_3^2 + 2t_4^2 - t_1 - t_4 - 10 \leq 0.
    \label{example1}
\end{align}

In this test example, $\uptheta = [t_1, t_2, t_3, t_4]^T$. The optimal point $\uptheta^* = [0, 1, 2, -1]^T$ and the constraints $q_1(\uptheta)$ and $q_2(\uptheta)$ are active at $\uptheta^*$. The Lagrange multiplier at $\uptheta^*$ is $\uplambda^* = [2,1,0]$. We add i.i.d. noise following normal distribution $N(0, 4)$ to the objective function and choose the initial point $\hat{\uptheta}_0$ as $[-2, -2, -2, -2]^T$, which is outside the feasible set. We set $\upalpha=0.602$, $\upgamma=0.101$, $\upbeta=1$, $a=0.1$, $A=100$, and $c=1$. For the QP algorithm, we set $r_k=2(k+1)^{0.1}$. For the AL algorithm, we set $r_k=(k+1)^{0.1}$ and the initial value of $\uplambda_k$ as $[0,0,0]^T$. For the AVP algorithm, we set $r_k=r=3.5$.

We generate the results of the averaged relative error $\|\hat{\uptheta}_K-\uptheta^*\|/\|\hat{\uptheta}_0-\uptheta^*\|$ based on $50$ independent replicates with $4000$ loss function measurements ($K=2000$) in each replicate, and define a non-negative value $Q(\uptheta)=\sum_{i=1}^m\max\{0,q_i(\uptheta)\}/m$, using the averaged $Q(\hat{\uptheta}_K)$ to measure the amount of constraint violation of the final solutions. Smaller values of the averaged $Q(\hat{\uptheta}_K)$ means less infeasibility, while the value of zero indicates all feasible outputs. Besides, we also conduct three pairs of two-sample tests between the SU algorithm and each of the AVP, QP and AL algorithm, and each with the null hypothesis $H_0$: The averaged relative error of the SU algorithm is larger or equal than the compared algorithm. A smaller $p$-value means that we are more confident to reject $H_0$. Additionally, we compute the proportion of the iterations using loss function measurements in the last $100$ iterations. Below are the results.

\begin{figure}[htbp]
    \begin{subfigure}{0.42\textwidth}
        \centering
        \includegraphics[width=\textwidth]{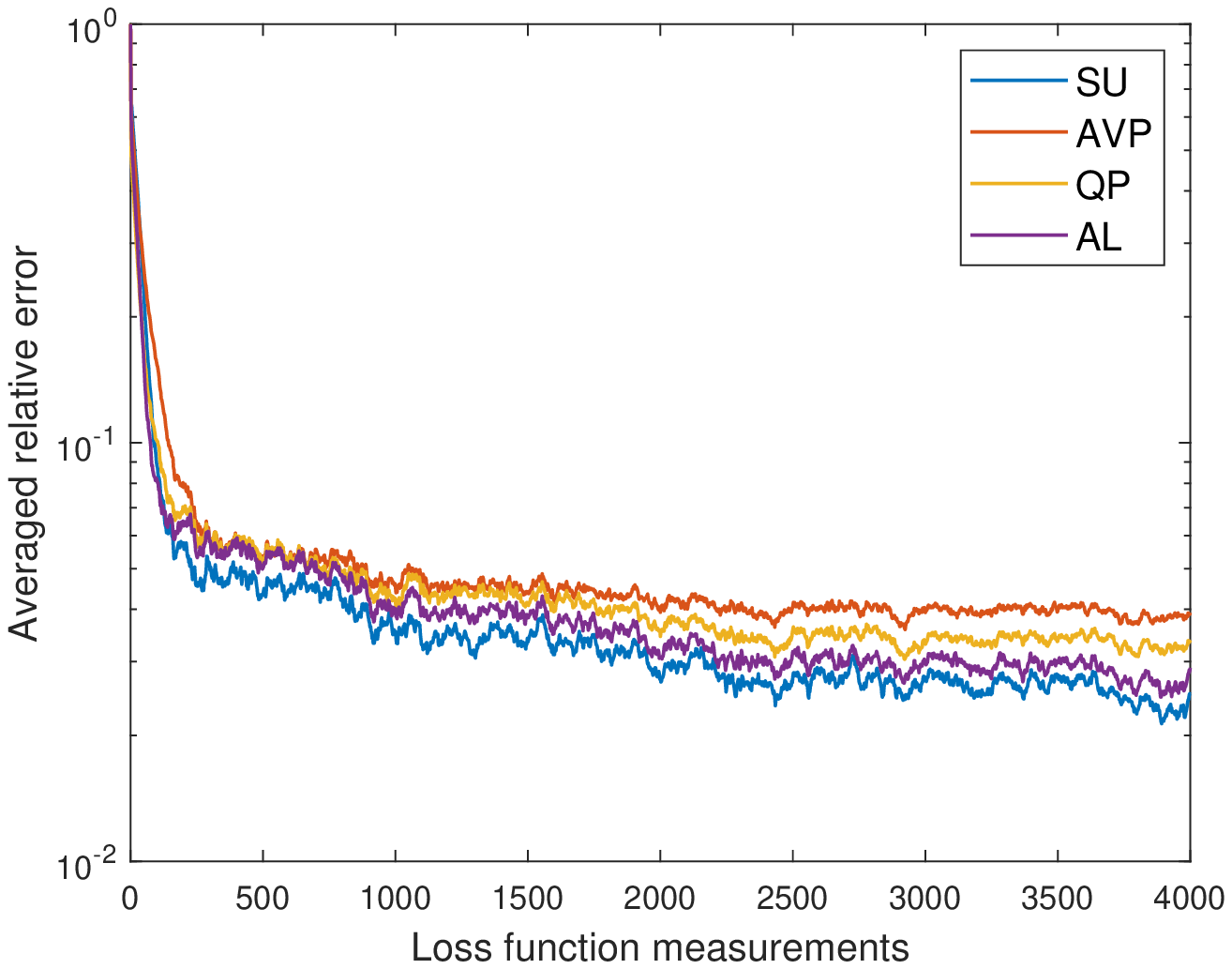}
        \caption{Averaged Relative Error}
    \end{subfigure}
    \hfill
    \begin{subfigure}{0.42\textwidth}
        \centering
        \includegraphics[width=\textwidth]{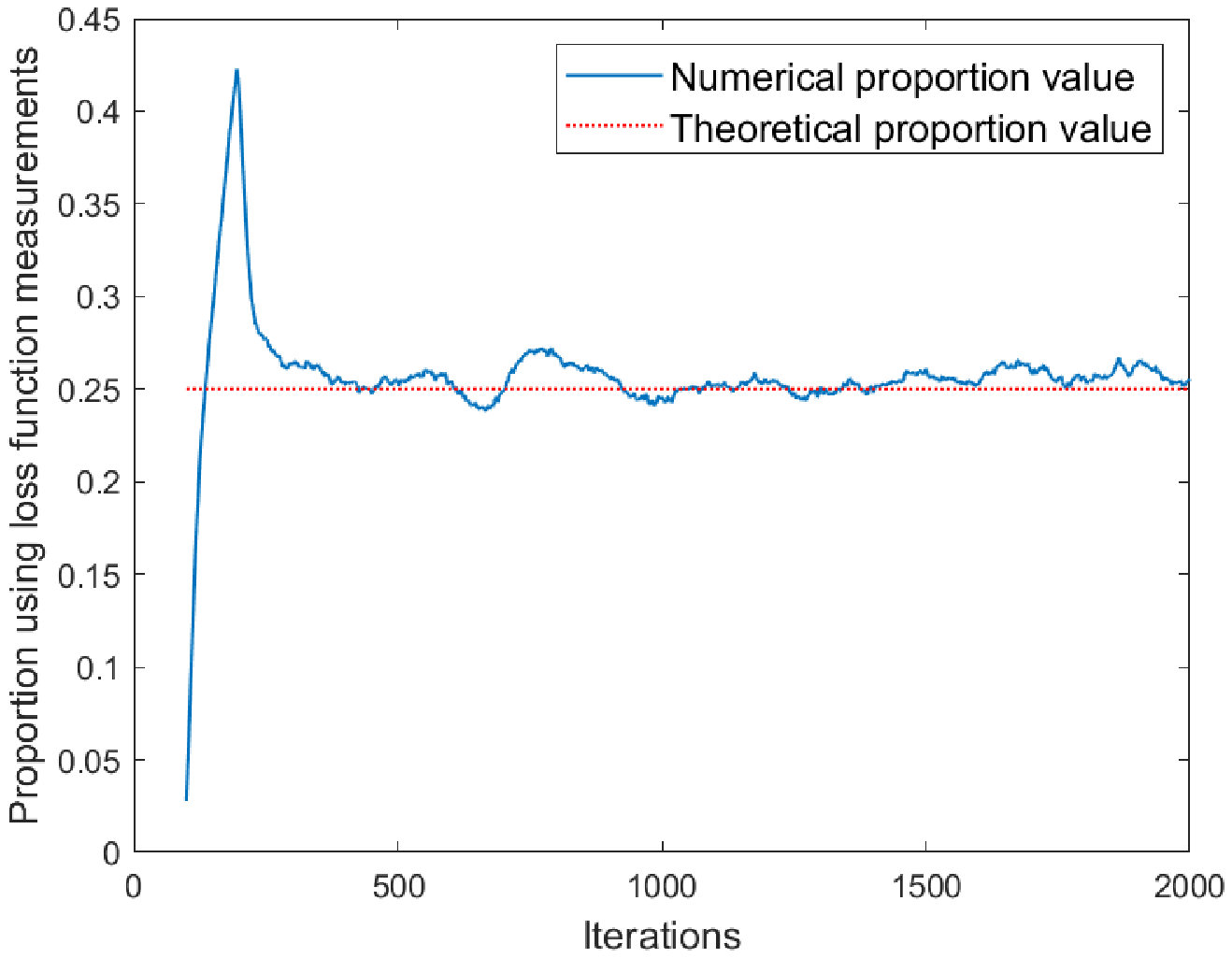}
        \caption{Asymptotic Proportion}
    \end{subfigure}
    \caption{The Quadratic Example}
    \label{quadratic}
\end{figure}

\begin{table}[htbp]
\begin{center}
\begin{tabular}{|c||c|c|c|}
\hline
 & $\|\hat{\uptheta}_K-\uptheta^*\|/\|\hat{\uptheta}_0-\uptheta^*\|$ & $p$-value versus SU & $Q(\hat{\uptheta}_K)$ \\
\hline
SU & $0.1374$ & NA & $0$ \\
\hline
AVP & $0.2149$ & $1.1083 \times 10^{-08}$ & $0.4988$ \\
\hline
QP & $0.1847$ & $5.0040 \times 10^{-04}$ & $0.2177$ \\
\hline
AL & $0.1587$ & $0.0638$ & $0.0498$ \\
\hline
\end{tabular}
\end{center}
\caption{Averaged Relative Error, Feasibility \\ and Two-sample Test Result}
\label{table_example1}
\end{table}

\subsection{A Quartic Example}
The second test example we use is a quartic case revised from our previous test example \eqref{example1}, which is also an example of Exercise 5.5 shown in \cite{spall2005introduction}:
\begin{align}
    \min\limits_{\uptheta} \ L(\uptheta) = &\sum_{i=1}^2 t_i^4 + \uptheta^T B \uptheta + \uptheta^TV, \nonumber\\
    \mathrm{s.t.} \ q_1(\uptheta) = &2t_1^2 + t_2^2 + t_3^2 + 2t_1 - t_2 - t_4 - 5 \leq 0, \nonumber\\
    q_2(\uptheta) = &t_1^2 + t_2^2 + t_3^2 + t_4^2 + t_1 - t_2 + t_3 - t_4 - 8 \leq 0, \nonumber\\
    q_3(\uptheta) = &t_1^2 + 2t_2^2 + t_3^2 + 2t_4^2 - t_1 - t_4 - 10 \leq 0,
    \label{example2}
\end{align}

where the values of the parameters are:

\begin{align}
    B = \begin{bmatrix} 0 & 0 & 3.5 & 0 \\ 0 & 1 & 0 & -8 \\ 3.5 & 0 & 8 & 0 \\ 0 & -8 & 0 & 5 \end{bmatrix}, V = \begin{bmatrix} -19 \\ -25 \\ -45 \\ 31 \end{bmatrix} + \upvarepsilon.
\end{align}

In this test example, $\upvarepsilon \in \mathbb{R}^{4}$ follows a multivariate normal distribution $N(\mathbf{0},4\mathrm{I}_{4})$, which makes the noise depend on $\uptheta$. The basic settings are the same as the quadratic example except that $K=3000$. Below are the results.

\begin{figure}[htbp]
    \begin{subfigure}{0.42\textwidth}
        \centering
        \includegraphics[width=\textwidth]{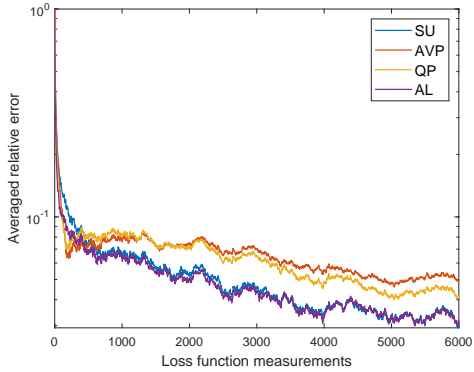}
        \caption{Averaged Relative Error}
    \end{subfigure}
    \hfill
    \begin{subfigure}{0.42\textwidth}
        \centering
        \includegraphics[width=\textwidth]{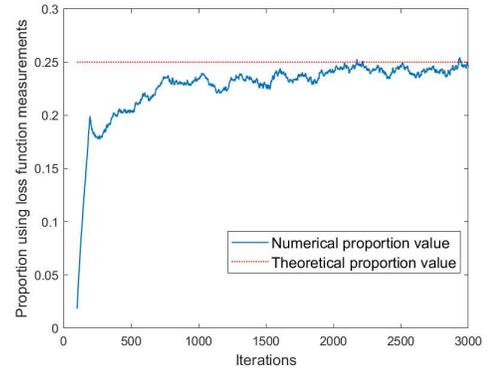}
        \caption{Asymptotic Proportion}
    \end{subfigure}
    \caption{The Quartic Example}
    \label{quartic}
\end{figure}

\begin{table}[htbp]
\begin{center}
\begin{tabular}{|c||c|c|c|}
\hline
 & $\|\hat{\uptheta}_K-\uptheta^*\|/\|\hat{\uptheta}_0-\uptheta^*\|$ & $p$-value versus SU & $Q(\hat{\uptheta}_K)$ \\
\hline
SU & $0.1718$ & NA & $0$ \\
\hline
AVP & $0.2703$ & $1.4572 \times 10^{-05}$ & $0.4244$ \\
\hline
QP & $0.2274$ & $0.0045$ & $0.1968$ \\
\hline
AL & $0.1740$ & $0.3878$ & $0.0291$ \\
\hline
\end{tabular}
\end{center}
\caption{Averaged Relative Error, Feasibility \\ and Two-sample Test Result}
\label{table_example2}
\end{table}

\section{Conclusions}
In this paper, we propose the SPSA-based switch updating algorithm for solving constrained stochastic optimization problems. It alternatively updates according to feasibility and conducts one step using the stochastic gradient of the objective loss function estimated by the SPSA algorithm or the gradient of one of the constraint functions. We also show the almost sure convergence result for our algorithm and the asymptotic result for the proportion of iterations that requires loss function measurements. Compared to the SPSA algorithms based on projection or penalty ideas, three dominant advantages exist for the SU algorithm:

(a) Simplicity for implementation and computing;

(b) Feasibility guarantees for final solutions;

(c) No necessity for extra hyper-parameter tuning.

Numerical experiments reveal the efficiency of the SU algorithm over the SPSA algorithm based on penalty ideas and justifies the asymptotic result. It could be postulated that under a large fixed number of loss function measurements, the SU algorithm achieves better performance than the SPSA algorithms based on penalty ideas, with its final solutions closer to the optimal solutions. Together with the advantages above, the SU algorithm is a more efficient and reliable method than the SPSA algorithms based on penalty ideas from an all-round perspective.

Further directions for future research do exist. Initially, switch updating ideas could be applied to more first-order stochastic approximation methods. Additionally, with proper assumptions or algorithmic strategies dealing with the noises added on constraint functions, convergence results could be extended to constrained stochastic optimization problems with noisy constraints (see \cite[Chap.~5]{kushner2012stochastic}). Finally, more work related to the theoretical results of the convergence rate of the SU algorithm could be conducted.

    {\small 
    \bibliographystyle{unsrt}
    \bibliography{ref}
    }
    \appendix

\end{document}